\documentclass[12pt]{amsart}

\usepackage{latexsym, amsmath, amscd, amssymb, amsthm} 
\usepackage[T2A]{fontenc}
 \usepackage[cp1251]{inputenc}
\newtheorem{te}{Theorem}
\newtheorem{lm}{Lemma}

\begin{document}

\noindent

 \title{ Analogue  of Sylvester-Cayley formula  for invariants of $n$-ary form }

\author{Leonid Bedratyuk} \address{ Khmel'nyts'ky National University, Instytuts'ka st. 11, Khmel'nyts'ky , 29016, Ukraine}
\email {bedratyuk@ief.tup.km.ua}
\begin{abstract}
The number $\nu_{n,d}(k)$ of linearly independed homogeneous invariants of  degree $k$  for the $n$-ary form of degree $d$  is calculated. The following formula   holds
$$
\nu_{n,d}(k)=\sum_{s \in W} (-1)^{|s|} c_{n,d}\bigl(k,(\rho-s(\rho))^*\bigr),
$$
here $W$  is Weyl group of Lie algebra $\mathfrak{sl_{n}},$ $(-1)^{|s|}$ is the sign   of   the element $s \in W,$  ${\rho=(1,1,\ldots,1)} $  is  half the sum of the positive roots of $\mathfrak{sl_{n}},$   the  weight  $\lambda^*$  means the  unique dominant weight on  the orbit $W(\lambda)$  and   $c_{n,d}\bigl(k,(m_1,m_2,\ldots,m_{n-1})\bigr)$  is 
the number  of nonnegative integer solutions of the system of equations 
$$
\left \{
\begin{array}{l}
2\, \omega_1(\alpha)+\omega_2(\alpha)+\cdots +\omega_{n-1}(\alpha)=d\,k-m_1, \\
\omega_1(\alpha)-\omega_2(\alpha)=m_2, \\
\ldots \\
\omega_{n-2}(\alpha)-\omega_{n-1}(\alpha)=m_{n-1},\\
|\alpha|=k.
\end{array}
\right.
$$
Here $\omega_r(\alpha)=\sum_{i \in I_{n,d}} i_r \alpha_i, $  $I_{n,d}:=\{i=(i_1,i_2,\ldots, i_{n-1})  \in \mathbb{Z}_{+}^{n-1},   |i|\leq d \},$ $ |i|=i_1+\cdots+ i_{n-1}.$
\end{abstract}
\maketitle

\noindent

\noindent
{\bf 1.}
Let   be  $F_{d,n}$  the $\mathbb{C}$-space of  $n$-ary forms  of degree $d:$
$$ 
 \sum_{i \in I_{n,d}} \,a_{i} { d \choose i} \, x_1^{d-(i_1+\cdots i_{n-1})} x_2^{i_1} \cdots x_n^{i_{n-1}},
$$

where  $I_{n,d}:=\{i=(i_1,i_2,\ldots, i_{n-1})  \in \mathbb{Z}_{+}^{n-1},   |i|\leq d \},$ $ |i|=i_1+\cdots+ i_{n-1},$ $a_{i} \in \mathbb{C}$  and 

$${ d \choose i}:=\frac{d!}{i_1!\, i_2! \cdots i_{n-1}! (d-(i_1+\cdots i_{n-1}))!}.$$ 
Let us identify the algebra of polynomial function $\mathbb{C}[F_{d,n}]$  with the polynomial $\mathbb{C}$-algebra $A_{d,n}$  of the variables set  $\{a_{i}, i \in I_{n,d} \}.$ The natural action of the group $SL_n$ on  $F_{d,n}$  induces    the actions  of   $SL_n$ ( and  $\mathfrak{sl_{n}}$)   on the algebra $A_{d,n}.$ The corresponding ring of  invariants   $A_{d,n}^{SL_n}=A_{d,n}^{\mathfrak{sl_{n}}}$  is  called the ring of invariants for the $n$-ary form of degree  $d.$

The ring  $A_{n,d}^{\mathfrak{sl_{n}}}$  is  graded ring 
$$
A_{d,n}^{\mathfrak{sl_{n}}}=(A_{d,n}^{\mathfrak{sl_{n}}})_0+(A_{d,n}^{\mathfrak{sl_{n}}})_1+\cdots+(A_{d,n}^{\mathfrak{sl_{n}}})_k+ \cdots,
$$
here  $(A_{d,n}^{\mathfrak{sl_{n}}})_k$ is the vector subspace generated  by   homogeneous  invariants of degree $k.$

Denote  $\nu_{n,d}(k) :=\dim (A_{d,n}^{SL_n})_k.$  For the binary form the number   $\nu_{2,d}(k)$  is  calculated  by well-known  Sylvester-Cayley formula, see  \cite{Hilb}.  For  the tenary form the  number  $\nu_{3,d}(k)$ is  calculated in the paper  of the present author, see  \cite{ASM}.  In this  paper we generalize  those formulas for the case  of $n$-ary form. 

{\bf 2.} In the Lie algebra    $\mathfrak{sl_{n}}$
denote  by  $E_{i\,j}$ the matrix unities.  The matrices  $H_1:=E_{2,\,2}{-}E_{1,\, 1},$ $H_2:=E_{3,\,3}{-}E_{2,\,2},$ $\ldots$ $H_{n-1}:=E_{n-1,\,n-1}{-}E_{n-2,\,n-2}$  generate the Cartan  subalgebra in   $\mathfrak{sl_{n}}.$ 

Recall that  the ordered set of   integer numbers  
$\lambda=(\lambda_1,\lambda_2,\ldots,\lambda_{n-1})$ is  called the weight of  $\mathfrak{sl_{n}}$-module $V,$  if there exists  $v \in V$ such that $v$ is common eigenvector of the operators $H_s$ and  $H_s(v)=\lambda_s \,v, $ $s=1,\ldots, n-1.$  A weight  is said to be  dominant weight if all  $\lambda_s \geq 0.$  Note,  our  definition of weight is  slightly different  than the standard  weight definition  as function on the Cartan subalgebra. Denote by $\Lambda_{V}$  the  set of all weight of $\mathfrak{sl_{n}}$-module  $V$  and denote by $\Lambda^{+}_{V}$  the set of dominant weight of $V.$ Also, denote by $\Gamma_{\lambda}$ the unique irreducible  $\mathfrak{sl_{n}}$-module  with highest  weight $\lambda.$

Let   $ A$ be the vector subspace of   $A_{n,d}$   generated by all elements $\{ a_{i}, i \in I_{d,n}\}$ of degree 1. 
The operators   $H_i,$ $i=1\ldots n-1$  act  on the basis elements   $a_i,$ $i \in I_{n,d}$ of the space $ A$ in  the  following  way, see   \cite{AA} :
$$
H_1(a_{i})=(d-(2\,i_1+i_2+\cdots +i_{n-1})) a_{i},  H_2(a_{i})=(i_2-i_3) a_{i}, \ldots H_{n-1}(a_{i})=(i_{n-2}-i_{n-1}) a_{i}.
$$

Therefore, the basis elements  $a_{i}, i \in I_{n,d}$ of   $\mathfrak{sl_{n}}$-module  $A$ are  the common eigenvectors  of the operators $H_s$  with the weights  
$$\varepsilon_{i}=\bigl(d-(2\,i_1+i_2+\cdots +i_{n-1}),i_2-i_3,\ldots,i_{n-2}-i_{n-1}\bigr).$$
 
It is clear that  $ A$   is  irreducible $\mathfrak{sl_{n}}$-module with  the highest  vector  $a_{(0,0,\ldots,0)}$   and with the highest  weight  $(d,0,\ldots,0).$  Thus,   $A \cong  \Gamma_{(d,0,\ldots,0)}.$ 
The number  $\nu_{n,d}(k)$  is the multiplicities  of trivial  $\mathfrak{sl_{n}}$-module  $\Gamma_{(0,0,\ldots,0)}$ in the decomposition the  symmetrical power   $S^k(A)$ on irreducible  $\mathfrak{sl_{n}}$-modules, i.e.  $\nu_{n,d}(k)=\gamma_{d,n}(k,(0,0,\ldots,0))$. Here  $\gamma_{d,n}(k,(0,0,\ldots,0))$  is determined from the decomposition:
$$
S^k(A)\cong \sum_{\lambda} \gamma_{n,d}(k,\lambda) \Gamma_{\lambda}, \lambda \in \Lambda^{+}_{S^k(A)}.
$$
Since the   multiplicities  of all  weights  are equal to 1 we  may write down the  the character of $A:$ 
$$
{\rm Char} (A)=\sum_{i \in I_{n,d}} e(\varepsilon_{i}).
$$
Here  $e(\varepsilon_{i})$ are generating elements of the group ring of the weight lattice  $\mathbb{Z}(\Lambda_{A}).$

We need the following technical lemma
\begin{lm}
The    monomial  $\prod_{i \in I_{n,d}} a_{i}^{\alpha_{i}}$ of  total degree $k$ is the weight vector of   $\mathfrak{sl_{n}}$-module  $S^k(A)$ with the  weight  $$(n\,d- (2 \omega_1(\alpha)+\omega_2(\alpha)+\cdots +\omega_{n-1}(\alpha)),   \omega_1(\alpha)-\omega_2(\alpha), \ldots, \omega_{n-2}(\alpha)-\omega_{n-1}(\alpha)),$$ де $\omega_s(\alpha)=\sum_{i \in I_{n,d}} i_s\, \alpha_{i},$     $|\alpha|:=\sum_{i \in I_{n,d}} \alpha_{i}=k.$
\end{lm}
\begin{proof}
Direct calculations.
\end{proof}

The  character of  $\mathfrak{sl_{n}}$-module  $S^k(A)$  is the complete symmetrical polynomial  $H_k$  of the variable set $e(\varepsilon_i)$, $ i \in  I_{n,d},$  see \cite{FH}. Therefore, we  have 
$$
{\rm Char} (S^k(A))=\sum_{|\alpha|=k} \prod_{i \in I_{n,d}}  e(\varepsilon_{i})^{\alpha_{i}}
$$
\begin{lm}
$$
{\rm Char} (S^k(A))=\sum_{\mu} c_{n,d}(k,\mu) e(\mu),  \mu \in \Lambda_{S^k(A)}
$$
here $c_{n,d}(k,\mu) :=c_{n,d}(k,(\mu_1,\mu_2,\ldots,\mu_{n-1}))$ is the number    of nonnegative integer solutions of the system of equations
$$
\left \{
\begin{array}{l}
2\, \omega_1(\alpha)+\omega_2(\alpha)+\cdots +\omega_{n-1}(\alpha)=k\, d-\mu_1, \\
\omega_1(\alpha)-\omega_2(\alpha)=\mu_2, \\
\ldots \\
\omega_{n-2}(\alpha)-\omega_{n-1}(\alpha)=\mu_{n-1},\\
|\alpha|=k.
\end{array}
\right.
$$
\end{lm}
\begin{proof}
Direct calculations,using Lemma 1.
\end{proof}

For arbitrary  $\mu \in \Lambda_{\Gamma_{\lambda}}$  denote by  $\mu^*$   the  unique dominant weight on  the orbit $W(\mu)$ of the  Weyl group $W.$ Such dominant weight exists and  unique, see.  \cite{Hum}.

On the   $\mathfrak{sl_{n}}$-module  $\Gamma_{\lambda}$  let us define  the value $E_{\lambda}$ in the following way
$$
E_{\lambda}=\sum_{s \in W} (-1)^{|s|} n_{\lambda}((\rho-s(\rho))^*).
$$
 Here  $\rho$ is  half the sum of the positive roots of Lie algebra $\mathfrak{sl_{n}},$ $n_{\lambda}(\mu)$  is  the   multiplicities of the weight $\mu$ in  $\Gamma_{\lambda}$  and   $|s|$  is the sign   of   the element $s \in W.$  Note, $n_{\lambda}(\mu)=0$  if $\mu \notin \Gamma_{\lambda}.$

The following lemma plays crucial  role in the calculation.

\begin{lm}
$$
E_{\lambda}=\left\{ 
\begin{array}{c}
1, \lambda=(0,\ldots, 0) ,\\
0,  \lambda \neq (0,\ldots, 0).
\end{array}
\right.
$$
\end{lm}
\begin{proof}
For  $ \lambda=(0,\ldots, 0),$  there is nothing to prove -- the   multiplicities of trivial weight in the trivial representation is equal to 1.

Suppose  now  $\lambda \neq (0,\ldots,0).$  We  use the following  recurrence formula, see \cite{NS},  for the multiplicities $n_{\lambda}(\mu)$ of the weight $\mu$ in the $\mathfrak{sl_{n}}$-module $\Gamma_{\lambda}:$   
$$
\sum_{s \in W } (-1)^{|s|} n_{\lambda}\bigl(\mu +\rho-s(\rho)\bigr)=0.
$$
Substituting  $\mu=(0,\ldots0)$  and  taking into account that the  multiplicities of the weights $\rho-s(\rho)$  and $(\rho-s(\rho))^*$  coincides we obtain the formula.
\end{proof}

Now we  are  ready to calculate the value $\nu_{n,d}(k)$
\begin{te}
The number   $\nu_{n,d}(k)$ of linearly independed homogeneous invariants of  degree  $k$   $n$-ary form  of degree  $d$ is  calculated  by the  formula
$$
\nu_{n,d}(k))=\sum_{s \in W} (-1)^{|s|} c_{n,d}\bigl(k,(\rho-s(\rho))^*\bigr). \eqno{(*)}
$$
\end{te}
\begin{proof}
The  number  $\nu_{n,d}(k))$  is equal to the multiplicities  $\gamma_{n,d}(k,(0,0,\ldots,0))$ of trivial representation  $\Gamma_{(0,0,\ldots,0)}$  in the symmetrical power  $S^k(A).$  The decomposition 

$$
S^k(A)\cong \sum_{\lambda} \gamma_{n,d}(k,\lambda) \Gamma_{\lambda}, \lambda \in \Lambda^{+}_{S^k(A)},
$$
 implies  the folowing characters decompositions
$$
{\rm Char}(S^k(A)) = \sum_{\lambda \in \Lambda^{+}_{S^k(A)}} \gamma_{n,d}(k,\lambda) {\rm Char}\bigl(\Gamma_{\lambda} \bigr).
$$
Taking into  account 
$$
{\rm Char}\bigl(\Gamma_{\lambda} \bigr)=\sum_{\mu \in \Lambda_{\Gamma_{\lambda}}} n_{\lambda}(\mu) e(\mu), 
$$
we get
$$
{\rm Char}(S^k(A))= \sum_{\lambda \in \Lambda^{+}_{S^k(A)}} \gamma_{n,d}(k,\lambda) \sum_{\mu \in \Lambda_{\Gamma_{\lambda}}} n_{\lambda}(\mu) e(\mu)= \sum_{\mu \in \Lambda_{S^k(A)} }\Bigl( \sum_{\lambda \in  \Lambda^{+}_{S^k(A)}} \gamma_{n,d}(k,\lambda) n_{\lambda}(\mu) \Bigr)e(\mu).
$$
By using Lemma 2  we  get 
$$
{\rm Char} (S^k(A))=\sum_{\mu \in \Lambda_{S^k(A)} } c_{n,d}(k,\mu) e(\mu).
$$
Therefore
$$
\sum_{\mu \in \Lambda_{S^k(A)} } c_{n,d}(k,\mu) e(\mu)= \sum_{\mu \in \Lambda_{S^k(A)} }\Bigl( \sum_{\lambda \in  \Lambda^{+}_{S^k(A)}} \gamma_{n,d}(k,\lambda) n_{\lambda}(\mu) \Bigr)e(\mu).
$$
By equating  the coefficients of  $e(\mu),$  we  obtain 
$$
 c_{n,d}(k,\mu) = \sum_{\lambda \in  \Lambda^{+}_{S^k(A)}} \gamma_{n,d}(k,\lambda) n_{\lambda}(\mu).
$$
Then, by using previous lemma  we have 
$$
\sum_{s \in W} (-1)^{|s|} c_{n,d}\bigl(k,(\rho-s(\rho))^*\bigr)=\sum_{s \in W} (-1)^{|s|}\sum_{\lambda \in  \Lambda^{+}_{S^k(A)}} \gamma_{n,d}(k,\lambda) n_{\lambda}((\rho-s(\rho))^*)=
$$
$$
=\sum_{\lambda \in  \Lambda^{+}_{S^k(A)}} \gamma_{n,d}(k,\lambda) \Bigl( \sum_{s \in W} (-1)^{|s|}n_{\lambda}((\rho-s(\rho))^*)\Bigr)=\sum_{\lambda \in  \Lambda^{+}_{S^k(A)}} \gamma_{n,d}(k,\lambda) E_{\lambda}=
\gamma_{n,d}(k,(0,0,\ldots,0)).
$$
Thus, 
$$
\nu_{n,d}(k)=\sum_{s \in W} (-1)^{|s|} c_{n,d}\bigl(k,(\rho-s(\rho))^*\bigr).
$$
This concludes the proof.
\end{proof}

\noindent
{\bf 3.}
It is easy to see that the multiplicities  $\gamma_{n,d}(k,\lambda)$  is  equal to number of linearly independed highest vectors of   $\mathfrak{sl_{n}}$-module $S^k(A)$ with    highest weight   $\lambda.$  Any highest vector  is the invariant of the subalgebra of upper triangular unipotent matrices of the Lie  algebra  $\mathfrak{sl_{n}}$.  Such elements is called semi-invariants of $n$-ary form.   In the same way we can  prove the theorem 
\begin{te} 
$$
\gamma_{n,d}(k,\lambda)=\sum_{s \in W} (-1)^{|s|} c_{n,d}\bigl(k,(\lambda+\rho-s(\rho))^*\bigr), \lambda \in  \Lambda^{+}_{S^k(A)}.
$$
\end{te}
Let us introduce the conception of covariants for  $n$-ary form.  Denote by  $C_{n,d}$  the  algebra of polynomial functions   of  the following  $\mathfrak{sl_{n}}$-module 
$$
A_{n,d}\oplus \Gamma_{(1,0,\ldots,0)} \oplus \Gamma_{(0,1,\ldots,0)} \cdots \oplus \Gamma_{(0,0,\ldots,1)}. 
$$
The subalgebra  $C_{n,d}^{\mathfrak{sl_{n}}}$ of  $\mathfrak{sl_{n}}$-invariants  is called the algebra of covarians of $n$-ary form.
For  $n=2$ the classical Roberts' theorem \cite{Rob} states  that the algebras of   semi-invariants and covariants are isomorphic.  The similar result for  ternary   form proved  by the  author in  the paper  \cite{AA}. 

The following statement seems to hold

\noindent
{\bf Conjecture.} {\it The algebras of semi-invariants and covariants  of $n$-ary form are isomorphic.}

If it is true,   then Theorem 2 defines the formula for calculation of the number linearly independed covariants for $n$-ary form of the weight $\lambda$ and degree  $k.$

\noindent
{\bf 4.}  {\bf Example.} Let  $n=3.$  Then  half the sum of the positive roots $\rho$ is equal to  $(1,1).$ The Weyl group of Lie  algebra  $\mathfrak{sl_{3}}$ is  generated by the  three reflections  $s_{\alpha_1},$ $s_{\alpha_2},$ $s_{\alpha_3},$ here $\alpha_1=(2,-1),$ ${\alpha_2=(-1,2)}$ i ${\alpha_3=(1,1)}$ are all positive  roots. The orbit  $W(\rho)$ consists of  $6$ weights  -- (1,1) and  
$$
\begin{array}{ll}
 s_{\alpha_1}(1,1)=(-1,2), & (-1)^{|s_{\alpha_1}|}=-1,\\
 s_{\alpha_2}(1,1)=(2,-1), &  (-1)^{| s_{\alpha_2}|}=-1,\\
 s_{\alpha_3}(1,1)=(-1,-1), &  (-1)^{| s_{\alpha_3}|}=-1,\\
s_{\alpha_1} s_{\alpha_3}(1,1)=(1,-2), &  (-1)^{|s_{\alpha_1} s_{\alpha_3}|}=1,\\
s_{\alpha_3} s_{\alpha_1}(1,1)=(-2,1), &  (-1)^{|s_{\alpha_3} s_{\alpha_1}|}=1,\\
\end{array}
$$
Therefore
$$
\rho-W(\rho)=\{ (0,0), (2,-1),(-1,2), (2,2), (0,3), (3,0)\}.
$$
Thus, taking into account the signs of the correponsing elements of the group W we  get the following  indentity for any dominant weight $\lambda$ of  the standard irreducible     $\mathfrak{sl_{3}}$-module $\Gamma_{\lambda}:$
$$
 n_{\lambda}(0,0) - n_{\lambda}(2,-1)- n_{\lambda}(-1,2) - n_{\lambda} (2,2)+ n_{\lambda}(0,3)+ n_{\lambda} (3,0)=0.
$$
Since the weights  $(1,1), (2,-1), (-1,2)$   lies on the same orbit   it is  implies  $(2,-1)^*=(-1,2)^*=(1,1)$ and  $$n_{\lambda}(1,1)=n_{\lambda}(2,-1)=n_{\lambda}(-1,2).$$
Thus
$$
 n_{\lambda}(0,0)- 2\,n_{\lambda}(1,1) - n_{\lambda} (2,2)+ n_{\lambda}(0,3)+ n_{\lambda} (3,0)=0.
$$
Therefore,  using Theorem  1,  we  obtain
$$
\nu_{3,d}(k)= c_{3,d}(k,(0,0))- 2\,c_{3,d}(k,(1,1)) - c_{3,d}(k, (2,2))+c_{3,d}(k,(0,3))+ c_{3,d}(k, (3,0)).
$$
It  coincides completelly with result of the paper  \cite{ASM}.

\noindent
{\bf 5.}  Let us derive the formula  for calculation of  $\nu_{n,d}(k).$  Solving the system of equations 
$$
\left \{
\begin{array}{l}
2\, \omega_1(\alpha)+\omega_2(\alpha)+\cdots +\omega_{n-1}(\alpha)=d\,n-\mu_1, \\
\omega_1(\alpha)-\omega_2(\alpha)=\mu_2, \\
\ldots \\
\omega_{n-2}(\alpha)-\omega_{n-1}(\alpha)=\mu_{n-1},\\
|\alpha|=k.
\end{array}
\right.
$$
for  $\omega_1(\alpha), \omega_2(\alpha), \ldots, \omega_{n-1}(\alpha) $  we  get  
$$
\left \{
\begin{array}{l}
\displaystyle \omega_1(\alpha)= \frac{k\,d}{n}-\Bigl(\frac{1}{n}\,(\mu_1+2\,\mu_2+\cdots+(n-1)\, \mu_{n-1})-(\mu_2+\mu_3+\cdots +\mu_{n-1})\Bigr)\\
\displaystyle \omega_2(\alpha)= \frac{k\,d}{n}-\Bigl(\frac{1}{n}\,(\mu_1+2\,\mu_2+\cdots+(n-1)\, \mu_{n-1})-(\mu_3+\mu_4+\cdots +\mu_{n-1})\Bigr) \\
\ldots \\
\displaystyle \omega_s(\alpha)= \frac{k\,d}{n}-\Bigl(\frac{1}{n}\,(\mu_1+2\,\mu_2+\cdots+(n-1)\, \mu_{n-1})-(\mu_{s+1}+\mu_{s+2}+\cdots +\mu_{n-1})\Bigr)\\
\ldots\\
\displaystyle \omega_{n-1}(\alpha)= \frac{k\,d}{n}-\Bigl(\frac{1}{n}\,(\mu_1+2\,\mu_2+\cdots+(n-1)\, \mu_{n-1})\Bigr)\\
|\alpha|=k
\end{array}
\right.
$$
It is not hard to prove that  the number  $c_{n,d}(k,(0,0,\ldots,0))$ of   nonnegative integer solutions of the following system 
$$
\left \{
\begin{array}{l}
\displaystyle \omega_1(\alpha)= \frac{k\,d}{n}\\
\displaystyle \omega_2(\alpha)= \frac{k\,d}{n} \\
\ldots \\
\displaystyle \omega_{n-1}(\alpha)= \frac{k\,d}{n}\\
|\alpha|=k
\end{array}
\right.
$$
is equal to coefficient of  $\displaystyle t^k (q_1 q_2 \ldots q_{n-1})^{\frac{k\,d}{n}}$    of the  expansion   of the  series

$$
R_{n,d}=\Bigl(\prod_{|\mu| \leq  d } (1-t q_1^{\mu_1} q_2^{\mu_2} \cdots q_{n-1}^{\mu_{n-1}})\Bigr)^{-1}.
$$
Denote it in such way:
$$c_{n,d}(k,(0,0,\ldots,0))=\Bigl(R_{n,d} \Bigr)_{t^k (q_1q_2 \cdots q_{n-1} )^{\frac{k\,d}{n}}}.$$
Then, for a set of integer numbers  $(\mu_1,\mu_2,\ldots,\mu_{n-1})$ the number  $c_{n,d}(k,(i_1,i_2,\ldots,i_{n-1}))$ of integer nonnegative solutions of the system of equations
 $$
\left \{
\begin{array}{l}
\displaystyle \omega_1(\alpha)= \frac{k\,d}{n}-\mu_1\\
\displaystyle \omega_2(\alpha)= \frac{k\,d}{n} -\mu_2\\
\ldots \\
\displaystyle \omega_{n-1}(\alpha)= \frac{k\,d}{n}-\mu_{n-1}\\
|\alpha|=k
\end{array}
\right.
$$
is  equals
$$c_{n,d}(k,(\mu_1,\mu_2,\ldots,\mu_{n-1}))=\Bigl( q_1^{\mu_1} q_2^{\mu_2}\cdots  q_{n-1}^{\mu_n} R_{n,d} \Bigr)_{t^k (q_1q_2 \cdots q_{n-1} )^{\frac{k\,d}{n}}}.$$

By using the multi-index notation rewrite the last expression in the  form 
$$
c_{n,d}(k,\mu)=\Bigl( q^{\mu}  R_{n,d} \Bigr)_{t^k (q )^{\frac{k\,d}{n}}}.
$$
To each  $\mu \in I_{n,d}$  assing  the following  vector 
$$
\mu_{\omega}=\left(\frac{1}{n}\,\Bigl(\sum_{s=1}^{n-1} s \mu_s\Bigr)-\sum_{s=2}^{n-1} \mu_s,\frac{1}{n}\,\Bigl(\sum_{s=1}^{n-1} s \mu_s\Bigr)-\sum_{s=3}^{n-1} \mu_s ,\,\dots, \frac{1}{n} \sum_{s=1}^{n-1} s \mu_s \right).
$$
Then the following formula holds
$$
\nu_{n,d}(k)=\left(  \Bigl( \sum_{s \in W} (-1)^{|s|} c_{n,d}\bigl(k,(\rho-s(\rho))^*\bigr) q^{(\rho-s(\rho))^*_{\omega}}  \Bigr)  R_{n,d}  \right)_{t^k (q )^{\frac{k\,d}{n}}}.
$$

\end{document}